\documentclass[12pt]{article}
\usepackage[a4paper]{geometry}
\usepackage{amsthm}

\usepackage{amssymb}
\usepackage{amsmath}
\usepackage{eufrak}

\newtheorem{theorem}{Theorem}

\newtheorem{corollary}[theorem]{Corollary}
\newtheorem{lemma}[theorem]{Lemma}

\newtheorem{definition}[theorem]{Definition}

\newtheorem{proposition}[theorem]{Proposition}

\begin{document}
\def\F{{\mathbb F}}
\title{ On pre-Lie rings related to some non-Lazard braces}
\author{ Agata Smoktunowicz}
\date{ }
\maketitle
\begin{abstract} 

Let $A$ be a brace of cardinality $p^{n}$ for some prime number $p$.  Denote $ann(p^{i})=\{a\in A: p^{i}a=0\}$. 
 Suppose that for  $i=1,2,\ldots $ and all $a,b\in A$ we have 
 \[a*(a*(\cdots *a*b))\in pA, a*(a*(\cdots *a*ann(p^{i})))\in ann(p^{i-1})\]  where $a$ appears less than $\frac {p-1}4$ times in this expression. 
 Let $k$ be such that $p^{k(p-1)}A=0$. It is shown that  the brace $A/ann(p^{4k})$  
is obtained from a left nilpotent  pre-Lie ring  by a formula which depends only on the additive group of brace $A$.  We also obtain some applications of this result. 
\end{abstract}

\section{Introduction}

 In \cite{Khukhro}  it was shown that if $L$ is a Lie ring  of cardinality $p^{n}$ for a prime number $p$,  natural number $n$ and nilpotency index less than $p$, 
 then Lazard's correspondence holds for the Lie ring $L$. In this paper we will investigate whether or not this result can be generalised  for  ad-nil Lie rings  of index less than ${\frac {p-1}4}$ that have left nilpotent pre-Lie rings. Observe that in  \cite{Bellester}  Ballester-Bolinches, Esteban-Romero, Ferrara,  Perez-Calabuig and  Trombetti gave an example of a  right nil brace which is not right nilpotent.
 On the other hand, by the result of Zelmanov,  finite Lie rings which are ad-nil of bounded index are nilpotent of a bounded index \cite{Zelmanov}. However,  it is not clear if this nilpotency index assures that Lazard's correspondence holds.  It is also known that every finite brace which is left nil is left nilpotent \cite{note}.  It is also known that a left nil pre-Lie algebra of index $ n$ over a field of characteristic zero is left nilpotent \cite{Filipov}.

This paper is also motivated by the following question: is there a generalisation of Lazard's correspondence 
  which works for Lie rings $L$ such that the factor Lie rings ${\frac {L}{pL}}$ are nilpotent rings of nilpotency index less than $ p$? In particular, if L is a Lie ring such that ${\frac {L}{pL}}$ is a nilpotent Lie ring of nilpotency index less than $p$, is there an analogon of the Baker-Campbel-Hausdorff formula which  produces a group on the set ${\frac {L}{p^{i}L}}$ for some $i$ (where $2i$ is less than the nilpotency index of $L$)?  A similar question can be posed about the group of flows \cite{AG}. Recall that the group obtained by the group of flows from a pre-Lie ring is isomorphic to the group obtained by the BCH formula applied to the corresponding Lie ring.

 We investigate these questions in the context of pre-Lie rings and braces. 
 In this paper we consider braces satisfying the following 
property $1'$\label{property1'}:
  Let  $A$ be a brace of cardinality $p^{n}$ for a prime number $p>2$ and a natural number $n$. We say that $A$ satisfies property $1'$ if  
$( a*(a\cdots *(a*b)\cdots ))\in pA$ where $a$ appears $\lfloor {\frac {p-1}4}\rfloor$ times in this expression. 
Moreover, we assume that 
$a*(a\cdots *(a*a)\cdots )\in pA$ where $a$ appears $\lfloor {\frac {p-1}4}\rfloor $ times in this expression  
 and  $\lfloor {\frac {p-1}4}\rfloor $ denotes the largest integer not exceeding ${\frac {p-1}4}$.
 We also consider braces satisfying the following Property $1''$: 
  Let  $A$ be a brace of cardinality $p^{n}$ for a prime number $p>2$ and a natural number $n$.  Denote $ann(p^{i})=\{a\in A: p^{i}a=0\}$. We say that $A$ satisfies property $1''$ if    
$a*(a\cdots *(a*(ann(p^{i}))))\in ann(p^{i-1}),$
  for $i=1,2, \ldots $ and for all $a\in A$, where $a$ appears $\lfloor {\frac {p-1}4}\rfloor$ times in this expression.

In this paper we associate Lie rings to some braces satisfying Properties $1'$ and $1''$. The multiplicative groups of such braces can have arbitrarily high nilpotency index, and in general they do not fit the context of Lazard's correspondence.  Our result is inspired by paper \cite{Shalev} where
 finite analogs of Lazard's $p$-adic Lie rings of $p$-adic Lie groups were constructed for finite groups outside of the context of Lazard's correspondence. In Theorem \ref{1} we associate pre-Lie rings to braces satisfying some special conditions by generalising results from \cite{asas}.  
Our main result is Theorem \ref{main}, which gives the passage from the obtained pre-Lie rings back to factor braces $A/ann(p^{i})$ for an appropriate $i$.

\begin{theorem}\label{main} Let $A$ be a brace of cardinality $p^{n}$ for some prime number $p$ and a natural number $n$. Suppose that $k$ is such that $p^{k(p-1)}A=0$.
Suppose that  $A$ satisfies properties $1'$ and $1''$.
 Then the brace $A/ann(p^{4k})$ is obtained from the left nilpotent  pre-Lie ring $(A/ann(p^{2k}), +, \bullet)$ obtained in Theorem \ref{1} by an algebraic formula which depends only on the additive group of $A$ (similar to the group of flows formula from \cite{AG}).
 \end{theorem}

 As a corollary of Theorem \ref{main}, we obtain 
\begin{corollary}\label{uniform} Let $p>3$ be a prime number. 
 Let $A$ be a brace whose additive group is a direct sum of some number of cyclic groups of cardinality $p^{\alpha }$ for some $\alpha $.
 Suppose that $A$ satisfies property $1'$.  
 Let $k\leq {\frac \alpha {p-1}}$.  
  Then the brace $A/ann(p^{4k})= A/p^{\alpha -4k}A$ is obtained from  the left nilpotent pre-Lie ring $(A/ann(p^{2k}), +, \bullet)$ obtained in Theorem \ref{1} by a formula which depends only on the additive group of $A$. 
\end{corollary}
  Corollary \ref{uniform} can be  applied to braces $A/p^{\alpha }A$ obtained in the following lemma.

\begin{lemma}\label{7} Let $p$ be a prime number.  Let $A$ be a brace whose additive group is isomorphic to $C_{p^{i_{1}}}\oplus \cdots \oplus C_{p^{i_{\beta }}}$ for some $\beta $ and some $i_{1}, \ldots , i_{t}$ and let $\alpha \leq i_{1}, \ldots , k\leq i_{t}$. Assume that $p^{\alpha }A$ is an ideal in $A$ (this happens for example when $A$ satisfies property $1'$), then the additive group of brace  $A/p^{\alpha }A$ is isomorphic to $C_{p^{\alpha }}^{\oplus \beta }=C_{p^{\alpha }}\oplus \cdots \oplus C_{p^{\alpha }}$. Consequently  the additive group of brace  
$A/p^{\alpha }A$ is uniform.
\end{lemma}
 In our proofs we apply methods from noncommutative ring theory.  Methods from noncommutative ring theory were also used in several other papers on related topics \cite{an, al, cjo, Rio, CO,  Facchini, IvanK, Iyudu, Dora}. In particular in \cite{Facchini} Michaela Cerqua and Alberto Facchini obtained very interesting results about  pre-Lie rings and braces, by using  methods from noncommutative ring theory. In \cite{KS} Kurdachenko and Subbotin developed  new methods to investigate braces with given  nilpotency index. 

Braces were introduced in 2007 by Wolfgang Rump \cite{rump} to describe all involutive non-degenerate set-theoretic solutions of the Yang-Baxter equation. 
The connection between pre-Lie rings and braces was first investigated  in \cite{Rump}, where a method for the construction of braces from left nilpotent pre-Lie rings was described.
 Some very interesting results about connections between braces and pre-Lie rings in the context of Lazard's correspondence and in the context of Lie groups  appeared recently in  \cite{ Rump, BA, doikou, Magnus, Sheng, ST}. 
 Most of these papers used an approach coming from areas such as algebraic geometry, Lie groups, Lie algebras, Lie rings and their cohomology, Hopf-Galois extensions, Magnus expansion and quantum integrable systems. In the case of finite braces the aforementioned results were obtained in the context  of Lazard's correspondence.

 In this paper we consider braces and pre-Lie rings whose left nilpotency index is arbitrarily high; these are usually outside of the context of Lazard's correspondence. We also apply some  ideas from \cite{L, Shalev, passage, asas}.

Notice that Lie and pre-Lie rings associated to groups and braces were used in various situations. Some very important applicatons of connection between  Lie rings and  groups appeared  in the papers of Kostrykin \cite{Kostrykin} and Zelmanov \cite{Zelmanov} in connection with the Burnside problem. Shalev used Lie rings associated to groups to investigate groups which have authomorphisms with a given number of fixed points \cite{Shalev}. Vaughan-Lee, Leedham-Green, Neuman, Wiegold  obtained  results on conjugacy classes in groups  (chapter 
$9$, \cite{Book}). Connections between Lie rings and groups, the Lazard's correspondence and Shur multipliers are used for calculating all groups of given cardinality \cite{Bettina}. Some applications of Lazard's correspondence are also mentioned in \cite{Kurdachenko, Khukhro, Efim}.

 We will now give an outline of this paper. Section $2$ contains basic background information on braces and pre-Lie rings.  In section $3$  we introduce braces satisfying property $1$ and show that $p^{i}A$ are ideals in such braces.
 In section $3$ we define a pullback function $\wp ^{-1}$ similarly as in \cite{asas}. This pullback function resembles division by $p^{k}$ for a given $k$.
 We also introduce function $f$, and prove some results about this function. In section $4$ we show that in  braces satisfying properties $1'$ and $1''$ sets $ann(p^{i})=\{a\in A: p^{i}A=0\}$ are ideals. We also prove some results about  braces satisfying properties $1'$ and $1''$. In section $5$ we associate pre-Lie rings to braces satisfying properties $1'$ and $1''$ (and also to some other types of braces). This is done in a similar way as in \cite{asas}, at the same time we use less restrictive assumptions than in \cite{asas}. We also  define operation $\odot $ for braces, in a simlar way as in \cite{asas}. In section $6$ we show that the obtained pre-Lie rings are left nilpotent, provided that the original braces satisfy properties $1'$ and $1''$. 
In section $7$ we investigate the passage brack from the obtained pre-Lie rings to factor braces $A/ann(p^{4})$, and prove the main results of the paper.

\section{Background information }  This section is similar to the introduction sections in \cite{passage, asas}. 
Recall that a   {\em pre-Lie ring} $(A, +, \cdot)$ is  an abelian group $(A,+)$ with a binary operation $(x, y) \rightarrow  x\cdot y$ such that 
\[(x\cdot y)\cdot z -x\cdot (y\cdot z) = (y\cdot x)\cdot z - y\cdot (x\cdot z)\]
and $(x+y)\cdot z=x\cdot z+y\cdot z, x\cdot (y+z)=x\cdot y+x\cdot z,$
 for every $x,y,z\in A$. A pre-Lie ring $A$  is {\em  nilpotent} or {\em strongly nilpotent}   if for some $n\in \mathbb N$ all products of $n$ elements in $A$ are zero. We say that $A$ is {\em left nilpotent} if for some $n$, we have $a_{1}\cdot (a_{2}\cdot( a_{3}\cdot (\cdots  a_{n})\cdots ))=0$ for all  $a_{1}, a_{2}, \ldots , a_{n}\in A$.

A set $A$ with binary operations $+$ and $* $ is a {\em  left brace} if $(A, +)$ is an abelian group and the following  version of distributivity combined with associativity holds.
  \[(a+b+a*b)* c=a* c+b* c+a* (b* c), \space  a* (b+c)=a* b+a* c,\]
for all $a, b, c\in A$; moreover  $(A, \circ )$ is a group, where we define $a\circ b=a+b+a* b$.
In what follows we will use the definition in terms of the operation `$\circ $' presented in \cite{cjo} (see \cite{rump}
for the original definition): a set $A$ with binary operations of addition $+$ and multiplication $\circ $ is a brace if $(A, +)$ is an abelian group, $(A, \circ )$ is a group and for every $a,b,c\in A$
\[a\circ (b+c)+a=a\circ b+a\circ c.\]
  All braces in this paper are left braces, and we will just call them braces.
 
Let $(A, +, \circ )$ be a brace.  Recall that $I\subseteq A$ is an ideal in $A$ if for
$i,j\in I$, $a\in A$ we have $i+j\in I, i-j\in I, i*a, a*i\in I$ where $a*b=a \circ b-a-b$.

Let $A$ be a brace, 
for any natural number $i$ we denote $p^{i}A=\{ p^{i}a: a\in A\}$. 

Let $A$ be a brace and $a\in A$ then we denote $a^{\circ j}=a\circ a\cdots \circ a$
  where $a$ appears $j$ times in the product on the right hand side.  By $A^{\circ p^{i}}$ we will   denote the subgroup of the multiplicative group $(A, \circ)$ of brace $A$ generated by $a^{\circ p^{i}}$ for $i\in A$. 
  In general for a set $S\subseteq A$, by $S^{\circ {p^{i}}}$ we denote the subgroup of $(A, \circ )$ generated by elements $s^{\circ {p^{i}}}$ for $s\in S$. 

 Let $(A, +, \circ )$ be a brace. One of the mappings used in connection with braces are the  maps $\lambda _{a}:A\rightarrow A$ for $a\in A$. Recall that  for $a,b, c\in A$ 
   we have $\lambda _{a}(b) = a \circ  b - b$, $\lambda _{a\circ c}(b) = \lambda _{a}(\lambda _{c}(b))$. 

\section{Braces  in which $p^{i}A$ are ideals}

 In this section we introduce braces which satisfy Property $1$. 

\begin{definition}\label{property1}
  Let  $A$ be a brace of cardinality $p^{n}$ for a prime number $p>2$ and a natural number $n$. We say that $A$ satisfies Property $1$ if  
\[a*(a*(a\cdots *(a*b)\cdots ))\in pA\] where $a$ appears ${\frac{p-1}2}$ times in this expression. 
Moreover, we assume that 
\[a*(a*(a\cdots *(a*a)\cdots ))\in pA\] where $a$ appears ${\frac{p-1}2}$ times in this expression. 
\end{definition}

 This property holds for example when 
$A^{\frac{p-1}2}\in pA$.

We recall Lemma $14$ from \cite{note}. 

\begin{lemma}\label{14} (Lemma $14$, \cite{note}).
 Let $(A, +, \circ )$ be a brace and let $a,b\in A$. Denote $a*b=a\circ b-a-b$. 
Denote $e_{1}'(a,b)=a*b$, $e_{i}(a)=a$ and inductively define 
\[e_{i+1}'(a,b)=a*e_{i}'(a,b), e_{i+1}(a)= a*e_{i}'(a).\] Then, for every natural number $j$:
\[a^{\circ j}=\sum_{i=1}^{j}{j \choose i} e_{j}'(a)\] 
\[ a^{\circ j}*b=\sum_{i=1}^{j}{j\choose i}\sigma _{j}e_{i}'(a,b).\]
\end{lemma}
 Our first result is as follows.
\begin{proposition}\label{333}
Let $A$ be a brace of cardinality $p^{n}$, where $p>2$ is a prime number and $n$ is a natural number. 
 Suppose that $(\lambda _{a}-I)^{\frac{p-1}2}(b)\in pA$ for all $a,b\in A$. 
 Then $p^{i}A$ is an ideal in brace $A$ for $i=1,2, \ldots $.

  Define  $E_{0}=A$, $E_{1}=A^{\circ p}$, and inductively $E_{i+1}=(E_{i})^{\circ p}$, so $E_{i+1}$ is the subgroup of $(A, \circ)$ generated by $p$-th powers of elements from $E_{i}$.Then $p^{i}A=E_{i}$ for $i=1,2, \ldots $.
\end{proposition}
\begin{proof} 
{\em Part 1.}

We will first show that if $c\in E_{i}$ then $c *b \in p^{i}A$ and $c\in p^{i}A$. 
 We proceed by induction on $i$. For $i=1$  we have $c\in E_{1}=A^{\circ p}$, so $c=a_{1}^{\circ p}\circ \cdots \circ a_{m}^{\circ p}$ for some $a_{1}, \ldots , a_{m}\in A$. The result then follows from Lemma $14$ in \cite{note} and from the fact that $\lambda _{a^{\circ p}}\lambda_{d^{\circ p}}=\lambda _{a^{\circ p}\circ d^{\circ p}}$ for $a,d\in A$ (and also because $(\lambda _{a}-I)^{p-1}(a)\in pA$ by assumption).

$ $

 Suppose that the result is true for some $i\geq 1$, so   for  all \[c\in E_{i}\] we have  \[c*b, c\in p^{i}A.\]
 We need to show that for all  $d\in  E_{i+1}$ we have $d*b\in p^{i+1}A$   and $d\in p^{i+1}A$. Observe that  it suffices to show that \[c^{\circ p}*b=\lambda _{c^{\circ p}}(b)-b\in p^{i}A\] and $c^{\circ p}\in p^{i}A$, because $\lambda _{a^{\circ p}}\lambda_{d^{\circ p}}=\lambda _{a^{\circ p}\circ d^{\circ p}}$, and because elements $c^{\circ p}$ for $c\in E_{i}$ generate $E_{i+1}$ (under the operation $\circ $).
 By Lemma \ref{14} we get 
$c^{\circ p}*b=\sum_{j=1}^{p}{p\choose j}e_{j}$ where $e_{1}=c*b$, and $e_{j+1}=c*e_{j}$.
Notice that $\sum_{j=1}^{p-1}{p\choose j}e_{j}\in p^{i+1}A$ since ${p \choose j}$ is divisible by $p$ for $0<j<p$ and $e_{j}\in p^{i}A$ since $c*b\in p^{i}A$ by the inductive assumption. So it suffices to show that $e_{p}\in p^{i+1}A$. Observe that \[e_{p}=c*(c*\cdots *e_{{\frac {p-1}2}})\in c*( \cdots *(c* (p^{i}A) ))\in p^{i}c*(\cdots c*(A))\subseteq p^{2i}A\subseteq p^{i+1}A,\] as required.

Similarly, observe that by Lemma $14$ from \cite{note} we get that $c^{\circ p}\in p^{i+1}A$.

$ $

{\em Part 2.}  We will now show that $p^{i}A\subseteq  E_{i}$ for $i=1,2, \ldots $.
 We will first proof Fact $1$. 

$ $

{\em Fact $1$.} Fix a natural number $i>0$. 
Denote \[G_{1}=p^{i-1}A\] 
 and inductively \[G_{j}=p^{i-1}A^{j}+p^{i}A.\]
 Assume that $c\in G_{j}$ then  \[(-c)^{\circ p}\circ (pc)\in pG_{j+1},\] for all $j$.  

{\em Proof of Fact 1.} For $j=1$ we have $c\in G_{1}=p^{i-1}A$ and
 we need to show that  $(-c)^{\circ p}\circ (pc)\in pG_{2}=p^{i}A^{2}+p^{i+1}A$.
 Lemma \ref{14} yields   
\[ (-c)^{\circ p}\circ (pc)=(-c)^{\circ p}+pc+ (-c)^{\circ p}*(pc)\in (-c)*(\cdots ((-c)*(-c)))+p^{i}A^{2}+p^{i+1}A\subseteq pG_{2}, \]
 where $c$ appears $p$ times in the product on the right hand side. This follows by 
 since  $ (-c)*(\cdots ((-c)* (-c)))\in ((-c)*\cdots *(p^{i-1}A))\subseteq p^{i-1}p^{2}A\subseteq pG_{2}$ (where $c$ appears $p$ times in the product).
 This shows that Fact $1$ is true for $j=1$.

Suppose that Fact $1$ holds for all numbers smaller than $j$, we then prove it for $j$.  
 Let $c\in G_{j}$. We need to show that  $(-c)^{\circ p}\circ (pc)\in pG_{j+1}$.
 Since $c\in G_{j}$, we have  \[c=p^{i-1}b+p^{i}e\] for some $b\in A^{j}, e\in A$. We calculate $(-c)^{\circ p}\circ (pc)=(-c)^{\circ p}+ pc  +(-c)^{\circ p}*(pc).$

 Note that \[(-c)^{\circ p}*(pc)=p^{i}((-c)^{\circ p}*b)+p^{i+1}((-c)^{\circ p}*e).\]
 Note that $p^{i+1}((-c)^{\circ p}*e)\in p^{i+1}A\subseteq pG_{j+1}$ and by Lemma \ref{14} we have   
$p^{i}((-c)^{\circ p}*(b))\in 
p^{i}A^{j+1}\subseteq pG_{j+1}$. 

 Reasoning similarly as above,  we have for $e_{1}=c$ and $e_{i+1}=(-c)*e_{i}$: 
 \[(-c)^{\circ p}+pc= \sum_{i=2}^{p-1}{p\choose i}e_{i}+e_{p}\in p^{i}A^{j+1}+p^{i+1}A\subseteq pG_{j+1},\] since $e_{1}\in p^{i-1}A^{j}+p^{i}A$. 
This concludes the proof of Fact $1$.

$ $

We will now show that $p^{i}A\in E_{i}$ for each $i$. It suffices to show that $p^{i}A\subseteq (p^{i-1}A)^{\circ p}$ for $i=1, 2, \ldots $.

We will show this by induction on $i$ in the reverse order, that is starting from $i=n$ and then proving it for $i=n-1, n-2 , \ldots $. 
 For $i=n$ we have $p^{n}A=0$ so $p^{n}A=0\subseteq (p^{i-1}A)^{\circ p}$.

 Let $0<i<n$ and suppose that the result holds for $i+1$, so \[p^{i+1}A\subseteq (p^{i}A)^{\circ p}.\]
 We need to  show that \[p^{i}A\in (p^{i-1}A)^{\circ p}.\]
  
Let $e\in p^{i}A$, then $e=p^{i}a$  for some $a\in A$, we will show that $e=c_{1}^{\circ p}\circ \cdots\circ c_{m}^{\circ p}$ for some $c_{1}, \ldots , c_{m}\in p^{i-1}A$. 

  Observe first that \[(-p^{i-1}a)^{\circ p}\circ (p^{i}a)=(-p^{i-1}a)^{\circ p}+ (p^{i}a)+ (-p^{i-1}a)^{\circ p}*(p^{i}a)= pa_{2}\] for some $a_{2}\in G_{2}$ by Fact $1$, since $p^{i-1}a\in p^{i-1}A=G_{1}$.

 Observe also that  $(-a_{2})^{\circ p}\circ (pa_{2})=pa_{3}$ for some $a_{3}\in G_{3}$ by Fact $1$. Continuing in this way we get elements $a_{3}\in G_{3}, a_{4}\in G_{4}, \ldots a_{n+1}\in G_{n+1}$. Observe that $G_{n+1}= p^{i-1}A^{n+1}+p^{i}A\in p^{i}A$, since $p^{n+1}A=0$ by Rump's result \cite{Rump}. Therefore, by the inductive assumption we have $pa_{n+1}\in
 p^{i+1}A\subseteq (p^{i}A)^{\circ p}$. 
 
 By the construction we get
 \[(-a_{n})^{\circ p}\circ (-a_{n-1})^{\circ {p}}\circ \cdots \circ (-a_{1})^{\circ p}\circ (pa_{1})
=pa_{n+1}\in (p^{i}A)^{\circ p}.\]

 Let $e_{i}\circ a_{i}=0$, so $e_{i}=a_{i}^{-1}\in G_{i}\subseteq p^{i-1}A$ then  
\[a\in  e_{n-1}^{\circ p}\circ \cdots \circ e_{1}^{\circ p}\circ (p^{i}A)^{\circ p}\in (p^{i-1}A)^{\circ p}.\] 
 This concludes the proof that  $p^{i}A\subseteq (p^{i-1}A)^{\circ p}$. Observe that this implies that $pA\subseteq A^{\circ p}=E_{1}$, $p^{2}A\subseteq E_{1}^{\circ p}=E_{2}$, and continuing in this way we get  $p^{i}A\subseteq (p^{-i-1}A)^{\circ p}=E_{i+1}^{\circ }= E_{i}$. 

$ $

{\em Part 3.} We will now show that $p^{i}A$ is an ideal in $A$ for each $i$.
 Note that $a*(p^{i}b)=p^{i}(a*b)$ and since $p^{i}A\subseteq E_{i}$ then $(p^{i}a)*b\subseteq p^{i}A$ by part 1 of this proof. 
\end{proof}

\subsection{Properties $1'$ and $1''$}

The aim of this section is to show that, with the assumptions from the previous section, it is also possible to prove that $p^{i}A=A^{\circ p^{i}}$. This  can  then be used in the subsequent sections.
 We also prove one similar property which is also useful for the subsequent applications.

 In this section we investigate braces which satisfy property $1'$ and property $1''$ which were introduced in the introduction section.

{\bf Proof of Lemma \ref{7}.} This follows from Proposition \ref{333}.

\subsection{ Some supporting lemmas  }\label{pullback}

In this section we slightly modify the definition of a pullback from \cite{asas}. The idea of applying the pullback function $\rho ^{-1}$ was inspired by papers \cite{L, Shalev} where another formula depending on $p$, namely $(a^{p}\circ b^{p})^{1/p}$, appeared in the context of group theory.
Let $A$ be a brace of cardinality $p^{n}$ for some prime number $p$ and some natural number $n$. Fix a natural number $k$.  
For $a\in p^{k}A$ let $\wp^{-1}(a)$ denote an element $x\in A$ such that $p^{k}x=a$. Such an element may not be uniquely determined in $A$, but we can fix for every $a\in p^{k}A$ such an element $\wp^{-1}(a)\in A$.
Notice that $p^{k}(\wp^{-1}(a))=p^{k}x=a$.

 For $A$ as above, $ann(p^{i})=\{a\in A:p^{i}A=0\}$. 

  Recall that if $I$ is an ideal in the brace $A$ then the factor brace $A/I$ is well defined.
The elements of the brace $A/I$ are cosets $[a]_{I}:= a + I=\{a+i:i\in I\}$ where $a \in A$, which we will simply denote by $[a]_{I}$, so $[a]_{I} =[b]_{I}$ if and only if $a-b\in I$.

\begin{lemma}\label{33} Let $A$ be a brace and let $I$ be an ideal in $A$. Let $A/I$ be defined as above. Assume that $ann(p^{k})\subseteq I$. 
Let $\wp^{-1} : p^{k}A \rightarrow A$ be defined as above. Then, for $a, b\in p^{k}A$ we have
$[\wp^{-1}(a)]_{I} + [\wp^{-1}(b)]_{I} = [\wp^{-1}(a + b)]_{I}.$ 
This implies that for any integer $m$ we have
 $[m \wp^{-1}(a)]_{I} = [\wp^{-1}(ma)]_{I}.$
\end{lemma}

\begin{proof} Note that $[\wp^{-1}(a)]_{I} + [\wp^{-1}(b)]_{I} = [\wp^{-1}(a + b)]_{I}$ 
       is equivalent to
   $\wp^{-1}(a)+ \wp^{-1}(b) - \wp^{-1}(a + b) \in I$,   so it suffices to show that 
   $\wp^{-1}(a)+ \wp^{-1}(b) - \wp^{-1}(a + b) \in ann(p^{k})$,
 which is equivalent to
   $p^{k}(\wp^{-1}(a)+ \wp^{-1}(b)-\wp^{-1}(a + b)) = 0.$ This in turn is equivalent to
   $a+b-(a+b) = 0,$ which holds true (since $p^{k}(\wp^{-1}(a)) = a$ for every $a\in p^{k}A$ by the definition of the function $\wp^{-1}$).
  \end{proof}

\subsection{Function $f$}
 Let $\wp^{-1}:p^{k}A\rightarrow A$ be defined as in Section \ref{pullback}. 
 Let $p$ be a prime number. Let $(A, +, \circ )$ be  a brace of cardinality $p^{n}$. Assume that $A$ satisfies property $1$. Let $a\in A$. Recall that 
  \[a^{\circ p^{k}}=\sum_{i=1}^{p^{k}}{{{p^{k}} \choose {i}}}e_{i}(a),\]
  where $e_{1}(a)=a$, $e_{2}(a)=a*a$ and $e_{i+1}(a)=a*e_{i}(a)$ for all $i$.
By Proposition \ref{333} we have $a^{\circ p^{k}}\in E_{k}\subseteq p^{k}A$. 

In this section we take for $f(a): A\rightarrow A$ to be any function such that   
\[ [f(a)]=[\wp ^{-1}(a^{\circ p^{k}})]\] where $\wp ^{-1}:p^{k}A\rightarrow A$ is defined as above.
 Then $[p^{k}f(a)]=[a^{\circ p^{k}}]$. 
 
In this section we will use the following notation:
\[[a]=[a]_{ann(p^{2k})}.\] 

\begin{theorem}\label{f(a)} Let $A$ be a brace. Assume that $ann(p^{i})$, $p^{i}A$ are ideals in $A$ for every $i$ and that $p^{i}A*ann(p^{j})\subseteq ann(p^{j-i})$ for all $j\geq i$.
 Suppose that $A$ satisfies property $1'$ and property $1''$. Then the map $[a]\rightarrow [f(a)]$ is well defined and injective. 
\end{theorem}
 \begin{proof} 
{\em  Part $a$}.  We will first  show that the map $[a]\rightarrow [f(a)]$ is well defined.
If $[a]=[b]$ then $a-b\in ann(p^{2k})$  hence $b=a\circ c$ for some $c\in ann(p^{2k})$.
 Recall that ${p}\choose j$ is divisible by $p$ for $j<p$.
  Let $e_{p}'(a), e_{p}'(b)$ be as in Lemma \ref{14}. By Lemma \ref{14} applied for $j=p$ and since $ ann(p^{2k})$ is an ideal in $A$ it follows that 
\[a^{\circ p}-b^{\circ p}-e_{p}'(a)+e_{p}'(b)\in p\cdot ann(p^{2k})\subseteq ann(p^{2k-1}).\]
 It follows since ${p\choose j}$ is divisible by $p$ for $0<j<p$. 

 By substituting $b=a\circ c$  in $e_{p}'(b)$ and using the fact which holds in any brace that $b*x=(a\circ c)*x=a*x+c*x+a*(c*x))$ we get that 
\[e_{p}'(b)-e_{p}'(a)\in e_{\frac {p-1}2}'(a, (e_{\frac{ p+1}2}'(b)-e_{\frac {p+1}2}'(a)))+ \sum  B*(\cdots <c>*( \cdots *(B*(e_{{\frac{ p+1}2}}'(b))))),\]
 where $<c>$ denotes the ideal generated by $c$ in $A$ and $B=\{a,c\}$.
 Observe that $e_{\frac{ p-1}2}'(b)-e_{\frac {p-1}2}'(a))\subseteq <c>\subseteq ann(p^{2k})$. 
By properties $1'$ and  $1''$ we have 
\[e_{p}'(b)-e_{p}'(a)\in e_{\frac {p-1}2}'(a, ann(p^{2k}))+p\cdot ann(p^{2k})\subseteq ann(p^{2k-1}),\]
 where $e_{\frac {p-1}2}'(a, ann(p^{2k}))$ is as in Lemma \ref{14}.
  Consequently, 
\[a^{\circ p}-b^{\circ p}\in  ann(p^{2k-1}).\]
By using the same argument again
for \[a'=a^{\circ p}, b'=b^{\circ p}, b'=a'\circ c'\]
for some $c'\in ann(p^{2k-1})$ we get that 
\[a^{\circ p^{2}}-b^{\circ p^{2}}=(a^{\circ p})^{\circ p}-(b^{\circ p})^{\circ p}\in  ann(p^{n-1})+p<c>\subseteq 
 ann(p^{2k-2}),\] where $<c>$ denotes the ideal generated by $c$ in brace $A$ (note that $c\in pA$ and $pA$ is an ideal in brace $A$).
 Continuing in this way we get \[a^{\circ p^{k}}-b^{\circ p^{k}}\in ann(p^{k}). \]
 Therefore $[a]=[b]$ implies $a^{\circ p^{k}}-b^{\circ p^{k}}\in ann(p^{k})$. This is equivalen to $p^{k}(f(a)-f(b))\in ann(p^{k})$, since $p^{k}f(a)=a^{\circ p^{k}}$. Observe that  $p^{k}(f(a)-f(b))\in  p^{k}\cdot ann(p^{2k})=ann(p^{k})\cap p^{k}A$ (indeed if $p^{k}c\in ann(p^{k})$ then $c\in ann(p^{2k})$).
 This implies that 
$[f(a)]=[f(b)]$, as required.

{\em Part $b$.} We will show that the map $[a]\rightarrow [f(a)]$ is injective.
 We will show for each $j\geq 0$ that if \[a-b\in p^{j}A+ann(p^{2k})\] and 
$[f(a)=[f(b)]$ then \[a-b\in p^{j+1}A+ann(p^{2k}).\] 
Note, that when applied for $j>n$ then $p^{j}A=0$ (this follows because the additive group $(A, +)$ of brace $A$ has cardinality $p^{n}$, so by the Lagrange theorem $p^{n}A=0$).
 This implies that  $[a]\rightarrow [f(a)]$
 is an injective function on $A/ann(p^{2k})$. We proceed by induction on $j$.

  $ $
 {\em Part $1$.} (Case of $j=0$)   Let $a,b\in A=p^{0}A+ann(p^{2k})$ and $f(a)-f(b)\in ann(p^{2k})$. We will show that $a-b\in pA+ann(p^{2k})$. 

Let $a,b\in A$. Suppose that $[f(a)]=[f(b)]$ then $f(a)-f(b)\in ann(p^{2k})$, hence 
 $p^{k}(p^{k}(f(a)-f(b))=0$ and $p^{k}f(a)-p^{k}f(b)\in p^{k}\cdot ann(p^{2k})$. 
 Observe that  $ p^{k}\cdot ann(p^{2k})$ is an ideal in $A$. It follows  because $ann(p^{k})$ and $p^{k}A$ are ideals in $A$ and 
$ p^{k}\cdot ann(p^{2k})=ann(p^{k})\cap p^{k}A$ (indeed if $p^{k}c\in ann(p^{k})$ then $c\in ann(p^{2k})$).

Recall that $p^{k}f(a)=a^{\circ {p^k}}$, hence 
 \[a^{\circ p^{k}}=b^{\circ p^{k}}+e',\] for some $e'\in p^{k}ann(p^{2k})$.
   We need to show that  $a-b\in pA+ann(p^{2k})$.

 We will first show that 
\[x_{1}*(x_{2}(*\cdots *(x_{t-1}*x_{t})\cdots ))-x_{1}'*(x_{2}'(*\cdots *(x_{t-1}'*x_{t}')\cdots ))\in pA +ann(p^{2k}),\]
 for all $x_{1}, \ldots , x_{t}, x_{1}', \ldots , x_{t}'\in\{a,b\}$. Indeed   for $t=1, x_{1}=a, x_{1}'=b$ we obtain the desired result that $a-b\in pA+ann(p^{2k})$.

 Observe that for $t=n+1$ the result holds as $A^{n+1}=0$ by Rump's result \cite{Rump}.  
 We proceed by induction on $t$, in the decreasing order, so we assume it holds for $t+1, t+2, \ldots $ and we will prove it for $t$.
 Let $y\in A$.  By Lemma \ref{14} we have:
 \[a^{\circ p^{k}}*y={{p^{k}}\choose 1}a*y+{ {p^{k}}\choose 2}a*(\cdots (a*y))+\cdots \] and  \[b^{\circ p^{k}}*y={{p^{k}}\choose 1}b*y+{ {p^{k}}\choose 2}b*(\cdots (b*y))+\cdots .\]

Observe moreover that for $l\geq p$ we have 
${ {p^{k}}\choose l}a*(\cdots (a*y))
\in p^{k+1}A$, where $a$  appears $l$ times in this expression (see additional explanations at the end of this proof).
Similarly,  
${ {p^{k}}\choose l}b*(\cdots (b*y))
\in p^{k+1}A$.
 Therefore, 
\[{ {p^{k}}\choose l}a*(\cdots (a*y))-{{p^{k}}\choose l}b*(\cdots (b*y))\in p^{k+1}A\] for $l\geq p$.
 Therefore, \[\sum_{i=1}^{p-1}{ {p^{k}}\choose l}(a*(\cdots (a*y))-
b*(\cdots (b*y)))\in a^{\circ p^{k}}*y-b^{\circ p^{k}}*y+ p^{k+1}A\subseteq
 p^{k+1}A+p^{k}ann(p^{2k}).\] 
It follows that \[\wp^{-1}(\sum_{i=1}^{p-1}{ {p^{k}}\choose l}(a*(\cdots (a*y))-
b*(\cdots (b*y))))\in pA+ann(p^{2k}).\]

This means that  \[(a*y+{\frac {p^{k}-1}2}a*(a*y)+\cdots )-(b*y+{\frac {p^{k}-1}2}b*(b*y)+\cdots )\in pA+ann(p^{2k}),\] where each sumation has $p-1$ summands.

 We now assume that $y=x_{1}*(x_{2}*(\cdots x_{t-2}*x_{t-1}))$ for some $x_{1}, \ldots , x_{t-1}\in \{a,b\}$.

 Notice that \[({\frac {p^{k}-1}2}a*(a*y)+\cdots )-({\frac {p^{k}-1}2}b*(b*y)+\cdots )\in pA+ann(p^{2k})\] by the inductive assumption on $t$, hence $a*y-b*y\in pA$.
 We can put elements $a^{\circ p^{k}}$ and $b^{\circ p^{k}}$ at a different place (not the first place) and combine the results to obtain  
\[x_{1}*(x_{2}(*\cdots *(x_{t-1}*x_{t})\cdots ))-x_{1}'*(x_{2}'(*\cdots *(x_{t-1}'*x_{t}')\cdots ))\in pA+ann(p^{2k}),\] for any $x_{i}, x_{i}'\in \{a,b\}$ for this given $t$ (similarly as it was done in unpublished paper \cite{newSmok}, on page $14$). 

$ $

{\em Part $2$.} (Case of general $j$)  Suppose that $j>0$ and $a-b\in p^{j}A+ann(p^{2k})$ and $[f(a)]=[f(b)]$ so $f(a)-f(b)\in ann(p^{2k})$.
 We will show that $a-b\in p^{j+1}+ann(p^{2k})$. 
Similarly as in case $j=0$ we denote $y=x_{1}*(x_{2}(*\cdots *(x_{t-2}*x_{t-1})\cdots ))$.  Observe that   $a^{\circ p^{k}}-b^{\circ p^{k}}\in p^{k}A\cap ann (p^{k})= p^{k}\cdot ann(p^{2k})$ gives
 \[a^{\circ p^{k}}*y=b^{\circ p^{k}}*y+e,\]
  for some $e\in p^{k}\cdot ann(p^{k})$, since $p^{k}\cdot ann(p^{2k})$ is an ideal in the brace $A$.
It follows that \[a^{\circ p^{k}}*y-b^{\circ p^{k}}*y\in p^{k}\cdot ann(p^{2k}).\]
 Therefore \[\sum_{i=1}^{p^{k}}{ {p^{k}}\choose l}(a*(\cdots (a*y))-b*(\cdots (b*y)))\in p^{k}\cdot ann(p^{2k}),\]
where $a$ and $b$ appear $l$ times at the expression near ${{p^{k}}\choose l}$.
Observe that that for $l\geq p$,  
\[{ {p^{k}}\choose l}a*(\cdots (a*y))-{{p^{k}}\choose l}b*(\cdots (b*y))\in p^{j+1+k}A+p^{k}ann(p^{2k})\] for $l\geq p$ (see additional explanations at the end of this proof).

 Therefore \[\sum_{i=1}^{p-1}{ {p^{k}}\choose l}(a*(\cdots (a*y))-
b*(\cdots (b*y)))\in p^{j+1+k}A+p^{k}\cdot ann(p^{2k}).\]  
It follows that \[\wp^{-1}(\sum_{i=1}^{p-1}{ {p^{k}}\choose l}(a*(\cdots (a*y))-
b*(\cdots (b*y))))\in p^{j+1}A+ann(p^{2k})\subseteq p^{j+1}A+ ann(p^{2k}).\] 
 The above two equations imply that $a*y-b*y\in p^{j+1}A+ann(p^{2k})$ (similarly as in Part $1$ above).
 
  By putting $a^{\circ p^{k}}$ and $b^{\circ k}$ at $j$th place instead of the first place (similarly as it was done in unpublished paper \cite{newSmok}, on page $14$), and by applying the same reasoning we obtain  
\[x_{1}*(x_{2}(*\cdots *(x_{t-1}*x_{t})\cdots ))-x_{1}'*(x_{2}'(*\cdots *(x_{t-1}'*x_{t}')\cdots ))\in p^{j+1}A+ann(p^{2k}),\] for any $x_{i}, x_{i}'\in \{a,b\}$ for this given $t$.

 This proves the inductive argument. Therefore $a-b\in p^{j}A+ann(p^{2k})$ for each $j\leq \alpha -2k$, consequently for $t=1$, $j=\alpha -2k$ we have  $[a]=[b]$. This concludes the proof.

$ $

{\em Additional explanations for parts  $1$ and $2$.}

{\em Explanations for part $1$.} We will show that  that for $l\geq p$ we have 
${ {p^{k}}\choose l}a*(\cdots (a*y))
\in p^{k+1}A$, where $a$  appears $l$ times in this expression. This follows from  Property $1$  and because ${ p^{k} \choose l}={{p^{k}-1}\choose {l-1}}\cdot {\frac {p^{k}}l}.$
 We will use notation from Lemma \ref{14}, we need to show that $p^{s_{l}+1}$ divides $e_{l}'(a,y)$, where $p^{s_{l}}$ divides $l$ and $s_{l}\geq 1$. Since $p^{s_{l}}$ divides $l$, we have $l\geq p^{s_{l}}$.
$e_{p^{s_{j}}}'(a,y)=$
$e_{p}'(a, e_{p}(a, \cdots e_{p}(a,y)))\in p^{s_{j+1}}A,$
 since $e_{p}'(a, p^{t}y)=p^{t}e_{p}'(a,y)\in p^{t+2}A$
 by Property $1'$ applied two times (for every $t$).
 Therefore, $e_{l}'(a,y)\in  p^{s_{j+1}}A$. 

$ $

{\em Explanations for part $2$}.
 Let $a-b\in p^{j}A+ann(p^{2k})$ for some $j>0$, we will show that 
${ {p^{k}}\choose l}(a*(\cdots (a*y))-b*(\cdots (b*y)))
\in p^{k+j+1}A+p^{k}ann(p^{2k})$, where $a$  appears $l$ times in this expression. We need to show that
 $p^{s_{l}+j+1}$ divides $e_{l}'(a,y)-e_{l}'(b,y)$ (where $s_{l}$ is the largest integer such that $p^{s_{l}}$ divides $l$).
  
Note that $a-b\in p^{j}A+ann(p^{2k})$  imply that $b=a\circ c$ for some $c\in p^{j}A+ann(p^{2k})$.

 Note that $e_{l}'(a,y)=e_{t}'(a, e_{l-t}(a,y))$ and $e_{l}'(b,y)=e_{t}'(b, e_{l-t}(b,y))$ for all $0<t<l$.   
Reasoning similarly as in part $1$ of this proof we have for every $j$
\[e_{l}'(b,y)-e_{l}'(a,y)\in \sum B*(\cdots *<c>(\cdots *(B*e_{t}'(b,y))))+e_{l-t}'(a, (e_{j}'(b,y)-e_{j}'(a,y))).\]
For $s_{l}=1$ we have $l\geq p$ and then for $t={\frac {p-1}2}$ we have 
\[e_{l}'(b,y)-e_{l}'(a, y)\in (p^{j+2}A+p^{2}ann(p^{2k}))+e_{l-t}'(a, <c>)\subseteq p^{j+2}A+p^{2}ann(p^{2k}) ,\]
 where $<c>$ denotes the ideal generated by $c$ in $A$.
 Therefore \[ {{p^{k}}\choose l}(a*(\cdots (a*y))-b*(\cdots (b*y)))\in p^{k-1}(p^{j+2}A+p^{2}ann(p^{2k}))\subseteq p^{k+j+1}A+p^{k}ann(p^{2k}).\]

If $s_{l}>1$ then reasoning in a similar way,  for $t={\frac {p-1}2}$, we have   
\[e_{l}(b, y)-e_{l, y}(a)\in  (p^{j+1+s_{l}}A+p^{s_{l}+1}ann(p^{2k}))+e_{t}'(a,  (e_{t}'(b,y)-e_{t}'(a,y))).\]
 It suffices to show that $e_{l-t}'(b,y)-e_{l-t}'(a,y))\in p^{s_{j}+j-1}A$, and it can be done be repeating the same argument again several times.

\end{proof}

 Let $A$ be a brace. Recall that $A^{\circ p^{k}}$ is the subgroup of the multiplicative group $(A, \circ )$ of brace $A$ by  $p^{k}$th powers of elements from $(A, \circ )$, so $A$ is generated by elements $a^{\circ p^{k}}$ for $a\in A$. 
\begin{corollary}\label{bezdowodu}
 Let assumption be as in Theorem \ref{f(a)}. Then, \[A^{\circ p^{k}}=\{a^{\circ p^{k}}:a\in A\}=p^{k}A.\]
\end{corollary}
 \begin{proof} By using an analogous proof as in 
  Theorem \ref{f(a)} it can be proved that $[f(a)]_{ann(p^{k})}\neq [f(b)]_{ann(p^{k})}$ if and only if $[a]_{ann (p^{k})}\neq [b]_{ann(p^{k})}$. This is equivalent to the fact that $p^{k}f(a)\neq p^{k}f(b)$ if and only if $p^{k}a\neq p^{k}b$.

Therefore the cardinality of the set  $\{a^{\circ p^{k}}:a\in A\}$ is the same as the cardinality of the set 
$p^{k}A=\{p^{k}a:a\in A\}$. Let $E_{k}$ be as in Proposition \ref{333}. 
 Observe that $\{a^{\circ p^{k}}:a\in A\}\subseteq A^{\circ p^{k}}\subseteq E_{k}\subseteq p^{k}A$
 where the last inclusion follows from Proposition \ref{333}. Therefore all these sets are equal.
\end{proof}

\section{ Braces in which $p^{i}A$ and $ann(p^{i})$ are ideals.}\label{x}

In this section we investigate properties of braces satisfying property $1'$ and $1''$.

\begin{lemma}\label{xyz}
Let $A$ be a brace of cardinality $p^{n}$, where $p$ is a prime number and $n$ is a natural number. 
 Suppose that $A$ satisfies property $1'$ and $1''$. Then $p^{i}A$ and $ann(p^{i})$ are ideals in brace $A$ for $i=1,2, \ldots $. 
 Moreover $p^{i}A=A^{\circ p^{i}}$ for every $i$.
 Moreover, $(p^{i}A)*ann(p^{j})\subseteq  ann(p^{j-i})$ for $j\geq i$.
\end{lemma}
\begin{proof}
{\em Part $1$.}  Notice that $p^{i}A$ are ideals in $A$ by Proposition \ref{333}. 
 We will show that $ann(p^{i})$ is an ideal in $A$ for every $i$.

 Notice that $a*(p^{i}b)=p^{i}(a*b)$. Therefore,  if $a\in  A$ and $b\in ann (p^{i})$ then $a*b\in ann (p^{i})$. 
   To show that $ann(p^{i})$ is an ideal it suffices to show that $a*b\in ann(p^{i})$ provided that $a\in ann(p^{i})$, $b\in A$.

Let $a\in ann(p^{i})$ for some $i$, observe that $a^{\circ p}\in ann(p^{i-1})$. Indeed, 
let notation be as in  Lemma \ref{14} 
then $a^{\circ p}=\sum_{i=1}^{p-1}{p\choose i}e_{i}'(a)+e_{p}'(a)$. Notice that ${p\choose i}$ is divisible by $p$ for $0<i<p$ so $p^{i-1}\sum_{i=1}^{p-1}{p\choose i}e_{i}'(a)=0$ so $\sum_{i=1}^{p-1}{p\choose i}e_{i}'(a)\in ann(p^{i-1})$.
  Observe also that $e_{p}'(a)\in ann(p^{i-1})$ by Property $1''$. 
 Therefore, $a^{\circ p}\in ann(p^{i-1})$. Applying this result several times we get \[a^{\circ p^{i}}=0.\]
 Consequently, $a^{\circ p^{i}}*b=0*b=0$.
Observe that   
$0=a^{\circ p^{i}}*b=a'^{\circ p}*b=\sum_{k=1}^{p} {p\choose  k}e_{k}'$ where $e_{1}'=a'*b$,  and $e_{j+1}'=a'*e_{j}'$ for each $j$, where $a'=a^{\circ p^{i-1}}$. 

 Note that $e_{{\frac {p-1}2}}'(a',b)\in pA$ by property $1'$, for every $a', b\in A$.
 Observe also that  \[e_{{\frac {p-1}2}}'(a',b)-
\rho ^{-1}( e_{{\frac {p-1}2}}'(a',pb))\in ann(p)\] for every $a', b\in A$. This can be seen by multyplying both sides by $p$.
 Therefore, \[e_{p}'(a',b)=e_{{\frac {p+1}2}}'(a',\rho ^{-1}( e_{{\frac {p-1}2}}'(a', pb))).\]
 This follows from Property $1''$. 

Let $a'=a^{\circ p^{i-1}}$, then for some integers $\sigma _{i}$ we have 
  \[0=a^{\circ p^{i}}*b=a'^{\circ p}*b=\sum_{i=1}^{p-1}\sigma _{i}e_{i}'(a',pb)+e_{{\frac {p+1}2}}'(a',\rho ^{-1}( e_{{\frac {p-1}2}}'(a',pb))).\]

 This implies, 
 \[e_{1}'(a',pb)= -\sum_{i=2}^{p}\sigma _{i}e_{i-1}'(a',e_{1}'(a',pb))-a'*(e_{{\frac {p-1}2}-1}'(a',\rho ^{-1}( e_{{\frac {p-1}2}-1}'(a',e_{1}'(a',pb)))).\]

 We can substitute the expression for $e_{1}'(a,pb)$ in the expressions at the right-hand side  of the above equation (at the end of each summand).

 Applying it several times we get $e_{1}'(a',pb)\in A^{n+1}=0$, hence $pe_{1}'(a',b)=0$. 
 This follows from the fact that $\rho ^{-1}$ ``can be moved to the right '' as long as there is a product of at most 
${\frac {p-1}4}$ elements before it on the left by property $1''$. 
Therefore, $pe_{1}'(a,b)=(a'*(pb))=0$.

If $a\in ann (p)$ then $a'=a^{\circ p^{0}}=a^{\circ 1}=a$ so $p(a*b)=0$ and hence $a*b\in ann(p)$ as required.
If $a\in ann(p^{i})$ for $i\geq 2$ denote $a''=a^{\circ p^{i-2}}$, and so $a''^{\circ p}=a'$.
 By the above, applied for $pb$ instead of $b$ and $a''$ instead of $a'$ we get 
\[0=a''^{\circ p}*(pb)=\sum_{i=1}^{p-1}\sigma _{i}e_{i}'(a'',p\cdot (pb))+e_{{\frac {p+1}2}}'(a',\rho ^{-1}( e_{{\frac {p-1}2}}'(a',p\cdot pb))).\]

 This implies, 
 \[e_{1}'(a',p^{2}b)= -\sum_{i=2}^{p}\sigma _{i}e_{i-1}'(a'',e_{1}''(a'',p^{2}b))-a'*(e_{{\frac {p-1}2}}'(a'',\rho ^{-1}( e_{{\frac {p-1}2}}'(a'',e_{1}'(a'',p^{2}b)))).\]

 We can substitute the expression for $e_{1}'(a,p^{2}b)$ in the expressions at the right-hand side  of the above equation. 
 Similarly as before, continuing to substitute in such way $n+1$ we get that $e_{1}'(a'', p^{2}b)\in A^{n+1}=0$.
 Continuing in this way we get $e_{1}'(a, p^{i}b)=0$, hence $p^{i}(a*b)=0$, and  so $a*b\in ann (p^{i})$.

{\em Part 2. } Observe that $p^{i}A=E_{i}$ by Proposition \ref{333}.
 We will now show that $(p^{i}A)*ann(p^{j})\subseteq  ann(p^{j-i})$ for $j\geq i$.
 Let $\lambda $ be defined as in the introduction (and also as at the beginning of the next section) then 
$\lambda _{ a^{\circ p^{i}}\circ c^{p^{i}}}(b) =
 \lambda _{a^{p}}(\lambda _{c^{p}}(b))$, hence 
\[(a^{\circ p}\circ c^{p})*b=
 (a^{\circ p}(* c^{p}*b)+ a^{\circ p}*b+ c^{p})*b.\]
  We also know that $p^{i}A=E_{i}=(E_{i-1})^{\circ p}$.
 Therefore, it suffices to show that for $a\in E_{i-1}$ we have  
\[a^{\circ p}*ann(p^{j})\subseteq ann(p^{j-i}).\] 
 We use induction on $i$ to prove this fact. 

 For $i=1$ we have $a\in E_{0}=A$ and by Lemma \ref{14}, 
\[a^{\circ p}*ann(p^{j})\subseteq \sum_{i=1}^{p-1}{p\choose i}e_{i}'(a, ann(p^{j}))+e_{p}'(a, ann(p^{j}))\subseteq ann(p^{j-1}),\] 
 by Property $1''$ and since ${p\choose i}$ is divisible by $p$ for $0<i<p$.

Suppose now that $i>1$. By the inductive assumption
$b*ann(p^{j})\subseteq ann(p^{j-i+1})$  for all $b\in E_{i-1}$. We need to show that 
 $b*ann(p^{j})\subseteq ann(p^{j-i})$  for all $b\in E_{i}$. 
 This is equivalent to show (by a similar argument as the begining of this proof) that for all 
$a\in E_{i-1}$ we have 
 \[a^{\circ p}*ann(p^{j})\subseteq ann(p^{j-i}).\]
  By Lemma \ref{14}  applied for $j=p$ we have 
\[a^{\circ p}*ann(p^{j})\subseteq \sum_{i=1}^{p-1}{p\choose i}e_{i}'(a, ann(p^{j}))+e_{p}'(a, ann(p^{j}))\subseteq ann(p^{j-1}).\] 
 By the inductive assumption $a*ann(p^{j})\subseteq ann(p^{j-i+1})$, hence  Similarly as before we get $\sum_{i=1}^{p-1}{p\choose i}e_{i}'(a, ann(p^{j}))\subseteq ann(p^{j-i})$.
 Observe also that $e_{p}'(a, ann(p^{j}))\subseteq e_{p-1}'(a, a*ann(p^{j}))\subseteq
 e_{p-1}'(a, ann(p^{j-i+1}))\subseteq ann(p^{j-i})$ by property $1''$. This concludes the proof.
\end{proof}

\section{Rings associative and nonassociative}
   Let $(A, +, \circ )$ be a brace. One of the mappings used in connection with braces are the  maps $\lambda _{a}:A\rightarrow A$ for $a\in A$. Recall that  for $a,b, c\in A$ 
   we have $\lambda _{a}(b) = a \circ  b - b$, $\lambda _{a\circ c}(b) = \lambda _{a}(\lambda _{c}(b))$.

\subsection{Generalising some results from other papers}

 In this section we will prove a  proposition which generalises Proposition $5.1$ from \cite{asas}.
We start with an easy  supporting lemma. Recall that for a given integer $m$ an integer $\xi $  is a primitive root modulo $m$ if every integer coprime to $m$ is
congruent to a power of $\xi $ modulo $m$. It is known that
there exists a primitive root modulo $p^j$ for every $j$ and every odd prime number $p$.
 \begin{lemma}\label{Engelxi}
  Let $p>2$ be a prime number and $n>1$ be a natural number.
Let $\xi=\gamma ^{p^{n-1}}$, where $\gamma $ is a primitive root modulo $p^{n}$. Then $\xi ^{p-1}$ is congruent to $1$ modulo $ p^{n}$.  Moreover,  $ \xi ^{j}$ is not congruent to $1$ modulo $p$ for any natural number $0<j<p-1$.
\end{lemma}
\begin{proof} The same proof as in \cite{passage} works.
\end{proof}
 We also have a supporting lemma:

\begin{lemma}\label{5}
 Let $A$ be a brace of cardinality $p^{n}$ for some $n$. Suppose that  $p^{i}A$ is an  ideal in $A$ for each $i$. Let $j_{1}, \ldots , j_{m}$ be natural numbers and let  $c_{i}\in p^{j_{i}}A$ for $i=1, \ldots , m$,  
 then  any product of elements  $c_{1}, \ldots ,  c_{m}$ and some elements from $A$ 
 belongs to $p^{j_{1}+\ldots +j_{m}}A$.
\end{lemma}
\begin{proof}
 It follows by applying several times the fact that  for each $a,b\in A$ there is $c\in A$ such that  $(p^{i}a)*b=p^{i}c$ (since $p^{i}A$ is an ideal in $A$), hence for each $j$ we have $(p^{i}a)*(p^{j}b)=p^{i+j}c$. 
\end{proof}
The following is a generalisation of Lemma $4.2$ and Corollary $4.3$ \cite{asas}.
\begin{lemma}\label{42} Let $W$ denote the set of all non-associative words in non-commuting variables $X,Y, Z$; moreover $Z$ appears only once at each word and both $X$ and $Y$ appear at least once in each word. Let $p$ be a prime number, let $n$ and $k$ be natural numbers. Then there are integers $\beta _{w}$ for $w\in W$, such that only a finite number of them is non-zero and the following holds: For each brace $(A, +, \circ )$ of cardinality $p^{n}$ and such that 
 $p^{i}A$ is an  ideal in $A$ for every $i$, and  $p^{(p-1)k}A=0$.
 Then for all $a,b\in p^{k}A$, $c\in A$ we have 
\[(a+b)*c=a*c+b*c+\sum_{w\in W} \beta _{w}w<a,b,c>,\]
 where $w(a,b,c)$ denote the specialisation of the word $w$ for $X=a, Y=b, Z=c$ and the multiplication in $w<a,b,c>$ is the same as the operation $*$ in brace $A$.
\end{lemma}
\begin{proof}
 The same argument as in the proof of Lemma $4.2$, \cite{asas} works. The following small changes need to be done to the proof of Lemma $4.2$ in \cite{asas}:  it should be written  ``products of any $n$ elements from the set $\{a,b\}$'' instead of products of $p$ elements from the set $\{a,b\}$, also the summation in the formulas should end at $2n$ instead of $2p$, and it should be added that the final results depends also on $n$ and $k$. In line $2$ it should be ``$a,b\in p^{k}A$'' instead of ``$a,b\in pA$''. In the last line Lemma \ref{5} should be used instead of Proposition $4.1$ \cite{asas}. 

\end{proof}

\begin{corollary}\label{43}
 Let $p$ be a prime number and let $m, n, k$ be natural numbers. Let $W'$ be the set of non-associative words in variables $X,Z$ where $Z$ appears only once at the end of each word and $X$ appears at least twice in each word. 
Then there are integers $\gamma _{w}$ for $w\in W'$, such that only a finite number of them is non-zero and the following holds: For each brace $(A, +, \circ )$ 
 of cardinality $p^{n}$ and such that $p^{i}A$ is an ideal in $A$ for every $i$,
and  $p^{(p-1)k}A=0$ 
 and for all $a\in p^{k}A$, $c\in A$ we have 
\[(a+b)*c=a*c+b*c+\sum_{w\in W} \beta _{w}w<a,b,c>,\]
 where $w(a,b,c)$ denote the specialisation of the word $w$ for $X=a, Y=b, Z=c$ and the multiplication in $w<a,b,c>$ is the same as the operation $*$ in brace $A$.
\end{corollary}
\begin{proof} 
 We can use the same proof as the proof of Lemma $4.2$, \cite{asas}. The following small changes need to be done to the proof of Lemma $4.2$ in \cite{asas}: 
 Lemma \ref{42} should be used instead of Lemma $4.2$ \cite{asas}. 
 Moreover, in the lines $6$, $9$,  $17$, $18$,   write $p^{k}A$ instead of $pA$. In  the line $22$ of the proof of Corollary $4.3$  \cite{asas} write 
``this is true  for $t\geq p$, since  $Y_{t}^{\alpha +1}\subseteq p^{kt}A= 0$, since $p^{pk}A=0$ by assumption.'' instead of  ``this is true for $t\geq p$ since $Y_{t}^{\alpha +1}\subseteq p^{p}A=0$.''

\end{proof}

 We are now ready to prove the main result of this section.
\begin{proposition}\label{12345}
Let $A$ be a  brace of cardinality $p^{n}$, where $p$ is a prime number and $n$ is a natural number. Assume that $p^{i}A$ is an ideal in $A$ for every $i$.  Let $\xi =\gamma  ^{p^{n-1}}$, where $\gamma $ is a primitive root modulo $p^{n}$.
 Let $k$  be a natural number such that $p^{(p-1)k}A=0$. 
 Define the binary operation $\cdot $ on $A$ as follows.
\[a\cdot b=\sum_{i=0}^{p-2}\xi ^{p-1-i}((\xi ^{i}a)* b),\]
    for $a, b\in A$.
 Then $(a+b)\cdot c=a\cdot c+b\cdot c, a\cdot (b'+c)=a\cdot b'+a\cdot c$
and $(a\cdot b)\cdot c-a\cdot (b \cdot c)=(b\cdot a)\cdot c-b\cdot (a\cdot c),$
for every $a, b\in p^{k}A$, $b', c\in A$.
In particular, $(p^{k}A, +, \cdot )$ is a pre-Lie ring.
\end{proposition}
\begin{proof}
 It is the same proof as the proof of Proposition $5.1$  in \cite{asas} with the following small changes : write $x,y\in p^{k}A$ instead of $x,y\in pA$  
  in  lines $2$, $8$, $15$, $18$, $20$, $27$, $35$, $44$, $48$  and  in Part $2$ of the proof  in  lines $3$, $14$, 
 $16$, $17$. Similarly, write $p^{k}A$ instead of $pA$ in Part $2$ of the proof in lines $25, 32, 38, 41, 61$.  At the end of line $35$ write ``$p^{k(p-1)}A=0$''  instead of 
 ``$p^{p-1}A=0$''. Also in line $39$ write in twp places $p^{n}$ instead if $p^{p}$. In line $46$ remove ``and $n<p-1$''. 
Additionally make the following small changes:
\begin{itemize}
\item Use Lemma \ref{5} from above instead of Proposition $4.1$ from \cite{asas}.
\item Use Lemma \ref{Engelxi} above instead of Lemma $3.5$ from \cite{asas}.
\item Instead of Lemma $4.2$ and Corollary $4.3$ from \cite{asas}   use Lemma \ref{42} and Corollary \ref{43}.
\end{itemize}
\end{proof}

\subsection{Pre-Lie rings related to braces in which $p^{i}A$ and $ann(p^{i})$ are ideals.}

The following result is a generalisation of Theorem $8.1$, \cite{asas}. We denote  $[a]=[a]_{ann(p^{2k})}$.
\begin{theorem}\label{1}
Let $(A, +, \circ )$ be a brace of cardinality $p^{n}$ for some prime number $p$ and some natural number $n$. Let $k$ be a natural number such that $p^{k(p-1)}A=0$. Assume that $p^{i}A$, $ann(p^{i})$ are ideals in $A$ and $(p^{k}A)*ann(p^{2k})\subseteq ann(p^{k})$. Let $\wp^{-1}:p^{k}A\rightarrow A$ be defined as in the section \ref{pullback}. Let $\odot $ be defined as 
\[[x]\odot [y]=[\wp^{-1}((p^{k}x)*y)].\] Let $\xi =\gamma  ^{p^{n-1}}$ where $\gamma $ is a primitive root modulo $p^{n}$.
 Define a binary  operation $\bullet $ on $A/ann(p^{2k})$ as follows:
\[[x]\bullet [y]=\sum_{i=0}^{p-2} \xi ^{p-1-i} [\xi ^{i}x]\odot [y],\]
for $x,y\in A$. Then $A/ann(p^{2k})$ with the binary operations $+$ and $\bullet $ is a pre-Lie ring.
\end{theorem}
\begin{proof} We first show that $\odot $ is a well defined binary operation on $A/ann(p^{2k})$, so that the result does not depend on the choice of coset representatives $x, y$ of cosets $[x], [y]$.
 We use an analogous proof as in Lemma $7.1$ \cite{asas}.
Observe that 
\[(p^{k}x)*y=(p^{k}(x-x')+(p^{k}x'))*y=(p^{k}x')*y+e ,\]
where $e$ belongs to the ideal generated in  the brace $A$ by the element $p^{k}(x-x')\subseteq p^{k}ann(p^{2k})=p^{k}A\cap  ann(p^{k})$ (indeed if $p^{k}c\in ann(p^{k})$ then $c\in
 ann(p^{2k})$ . By assumption $ann(p^{k}), ann(p^{k})$ are ideals in $A$.
It follows that  $e\in ann(p^{k})\cap p^{k}A=p^{k}ann(p^{2k})$, 
therefore $[x]\odot [y]-[x']\odot [y]=[\wp^{-1}(e)]\in [\wp^{-1}(p^{k}ann(p^{2k}))]\in [ann(p^{2k})]=[0]$ as required.

Observe  also that $[x]\odot [y]$ does not depend on the choice of the representative for the coset $[y]$. 
  Observe that \[[\wp^{-1}((p^{k}px)*y)]-[\wp^{-1}(( p^{k}x)*y')]=[\wp^{-1}((p^{k}x)*(y-y'))]=[0].\]
 This follows because $y-y'=ann(p^{2k})$ for some $c\in A$, and 
$(p^{k}x)*ann(p^{2k})\subseteq ann(p^{k})$ by assumptions.

  To show that $A/ann(p^{2k})$ with the binary operations $+$ and $\bullet $ is a pre-Lie ring we can use analogous proof as the proof on Theorem $8.1$ from \cite{asas} when we change $p$ to $p^{k}$ at each place, and $p^{2}$ to $p^{2k}$ and use Proposition \ref{12345} instead of Proposition $5.1$ from \cite{asas}. 
 We proceed as follows. 

$ $

{\em Part 1}.  We need to show that  $([x]+[y])\bullet [z]=[x]\bullet [z]+[y]\bullet [z]$ and
$[x]\bullet ([y]+[z])=[x]\bullet [y]+[x]\bullet [z]$ for every $x,y,z\in A$.  
Observe that the formula $[x]\bullet ([y]+[z])=[x]\bullet [y]+[x]\bullet [z]$ follows since $x*(y+z)=x*y+x*z$.
We will show now that \[([x]+[y])\bullet [z]=[x]\bullet [z]+[y]\bullet [z].\]
 Notice that $[x]\bullet [y]=[\wp ^{-1}(p^{k}x\cdot y)]$, where operation $\cdot $ is as in Proposition \ref{12345}. 
 By Proposition \ref{12345} and Lemma \ref{33}, 
\[([x]+[y])\bullet [z]=[\wp ^{-1}((p^{k}x+p^{k}y)\cdot z)]=[\wp ^{-1}(((p^{k}x)\cdot z)+((p^{k}y)\cdot z))]=\]
\[=[\wp ^{-1}((p^{k}x)\cdot z)]+[\wp ^{-1}((p^{k}y)\cdot z)]=[x]\bullet [z]+[y]\bullet [z].\]
\medskip

{\bf Part 2.} We need to show that \[([x]\bullet [y])\bullet [z]-[x]\bullet ([y]\bullet [z])=([y]\bullet [x])\bullet [z]-[y]\bullet ([x]\bullet [z]),\]
for $x, y,z\in A$. Recall that $[x]\bullet [y]=[\wp^{-1}((p^{k}x)\cdot y)].$ 
  Consequently  \[([x]\bullet [y])\bullet [z]=[\wp^{-1}((p^{k}x)\cdot y)]\bullet[z]=[\wp^{-1} (((p^{k}x)\cdot y)\cdot z)],\] and
\[[x]\bullet ([y]\bullet [z])=[x]\bullet [\wp^{-1}((p^{k}y)\cdot z) ]=[\wp^{-1}((p^{k}x)\cdot \wp^{-1}((p^{k}y)\cdot z))].\]

We need to show that \[[\wp^{-1} (((p^{k}x)\cdot y)\cdot z)]= [\wp^{-1}((p^{k}x)\cdot \wp^{-1}((p^{k}y)\cdot z))].\]
 It suffices to show that $\wp^{-1} (((p^{k}x)\cdot y)\cdot z)-\wp^{-1}((p^{k}x)\cdot \wp^{-1}((p^{k}y)\cdot z))\in ann(p^{2k})$. 
 We proceed in the same way as in \cite{asas}. 
Note that   \[p^{2k}(\wp^{-1}((p^{k}x)\cdot \wp^{-1}((p^{k}y)\cdot z)))=p^{k}((p^{k}x)\cdot \wp^{-1}((p^{k}y)\cdot z))=(p^{k}x)\cdot ((p^{k}y)\cdot z),\]

On the other hand we have
$p^{2k}(\wp^{-1} (((p^{k}x)\cdot y)\cdot z))=p^{k}(((p^{k}x)\cdot y)\cdot z).$
Note that $(p^{k}x)\cdot  y\in p^{k}A$, since, by the assumption $p^{k}A$ is an ideal in $A$.

Let $a\in p^{k}A$, $b\in A$, 
then $(na)\cdot b=n(a\cdot b)$ by Proposition \ref{12345}.
Applying this for $n=p^{k}$, $a=(p^{k}x)\cdot y$ and  $b=z$, we get that \[p^{k}(((p^{k}x)\cdot y)\cdot z)=(p^{k}((p^{k}x)\cdot y))\cdot z=((p^{k}x)\cdot (p^{k}y))\cdot z.\]
 Therefore,  \[(\wp^{-1}((p^{k}x)\cdot \wp^{-1}((p^{k}y)\cdot z)))-(\wp^{-1} (((p^{k}x)\cdot y)\cdot z))\in ann(p^{2k}).\]

By Proposition \ref{12345},  \[([x]\bullet[y])\bullet [z]-[x]\bullet  ([y]\bullet [z])=([y]\bullet [x])\bullet [z]-[y]\bullet ([x]\bullet [z]).\]
\end{proof}
\begin{corollary}
 Let $A$ be a brace satysfying preperties $1'$ and $1''$. Then $A$ satisfies assumptions of Theorem \ref{1}. 
\end{corollary}

 {\bf Remark.} Let $(A, +)$ be an additive group  of cardinality $p^{n}$ for a prime number $p$ and for a natural number $n$. Let $k$ be a natural number and let 
$(p^{k}A, \cdot, +)$ be a pre-Lie ring,  then $(A/ann (p^{4k}), \bullet, +)$ is a pre-Lie ring where we define for $a,b\in A$:  
 \[[a]_{ann(p^{4})}\bullet [b]_{ann(p^{4})}=[\wp^{-1}(\wp^{-1}((p^{k}a)\cdot (p^{k}b))]_{ann(p^{4k})}\]
 where $\wp^{-1}:p^{k}A\rightarrow p^{k}A$ is defined as in Section \ref{pullback}.
 The operation $+$ in the pre-Lie ring $(A/ann (p^{4k}), \bullet, +)$ is the same as in the factor group $A/ann(p^{4k})$. This may give a shorter proof of a less strong version of  Theorem \ref{1}. Notice however that in Theorem \ref{1}  we also proved that the operation $\odot $ is well defined, which we will  use later.

\section{Nilpotency of obtained pre-Lie rings}
 In this section we will show that  the pre-Lie rings constructed in Theorem \ref{1} 
 are left nilpotent.
 
  Recall that pre-Lie ring $A$  is {\em  nilpotent}   if for some $n\in \mathbb N$ all products of $n$ elements in $A$ are zero (nilpotent pre-Lie rings are also  called {\em strongly nilpotent}). We say that a pre-Lie ring $A$ is {\em left nilpotent} if there is $n$ such that $x_{1}\cdot (x_{2}\cdots  \cdot (x_{n-1}\cdot x_{n})\cdots )=0$ for all  $x_{1}, x_{2}, \ldots , x_{n}\in A$.

 \begin{proposition}\label{2} 
Let $(A, +, \circ )$ be a brace of cardinality $p^{n}$for some prime number $p$ and some natural number $n$. Let notation and  assumptions be as in  Theorem \ref{1}.
 Assume that   $(p^{k}A)*A^{j}\subseteq p^{k}A^{j+1}$ for $j=1,2, \ldots , n$.
 Then the following holds:
\begin{enumerate}
\item The pre-Lie ring  $P=(A/ann(p^{2k}), \bullet , +)$ constructed in Theorem \ref{1} is left nilpotent.
\item  Moreover, if brace $A$ satisfies $A^{c}\subseteq  pA$ for some $c$ then the pre-Lie ring $P$ satisfies $P^{c}\subseteq pA$ for the same $c$.  
\end{enumerate}  
\end{proposition} 
  \begin{proof}   Let $b\in A^{j}$ for some $j$, 
  $(p^{k}a) *b \in p^{k}A^{j+1}$ by assumptions.
   It follows that  $[\wp^{-1}((p^{k}a) * b)] \in [A^{j+1}]$.
       Therefore, by the formula for operation $\bullet $ from Theorem \ref{1}, we get 
 \[[a]\bullet [b] \in [A^{j+1}],\] for  $b\in  A^{j}$, $a\in A$. Consequently,
\[[b_{1}] \bullet  ([b_{2}]\bullet ( \cdots  ([b_{n}] \bullet  [b_{n+1}]) ))\in [A^{n+1}]= 0,\] for all $b_{1}, . . . b_{n+1}\in  A$, since  $A^{n+1}=0$ in every brace of cardinality $p^{n}$ for a prime $p$, by a result of Rump \cite{rump}.

 Notice that this immediately implies that $P^{c}\in pA$ provided that   
 the brace $A$ satisfies $A^{c}\in pA.$

\end{proof} 
\begin{proposition}\label{helo} Let $(A, +, \circ )$ be a brace of cardinality $p^{n}$for some prime number $p$ and some natural number $n$. Suppose that $A$ satisfies properties $1'$ and $1''$. Then $A$ satisfies assumptions of Theorem \ref{1}.
 Moreover, $(p^{i}A)*A^{j}\subseteq p^{i}A^{j+1}$ for $i, j=1,2, \ldots , n$.
 Consequently, $A$ satisfies assumptions of Proposition \ref{2}.
\end{proposition}
\begin{proof}
 By Lemma \ref{xyz} brace $A$ satisfies assumptions of Theorem $1$. It remains to show that 
 $(p^{i}A)*A^{j}\subseteq p^{i}A^{j+1}$ for $j=1,2, \ldots , n$.
By Lemma  \ref{xyz}, $p^{k}A=A^{\circ p^{k}}$. Moreover, by Lemma \ref{bezdowodu}, $p^{k}A=\{a^{\circ p^{k}}: a\in A\}$. 
It suffices to show that 
 $a^{\circ p^{i}}*A^{j}\subseteq p^{i}A^{j+1}$ for $j=1,2, \ldots , n$ for $i=1,2, \ldots $.
  For $i=1$ it follows from Lemma \ref{14} applied for $j=p$. 
 Suppose that the result holds for some $i$, so 
  $a^{\circ p^{i}}*A^{j}\subseteq p^{i}A^{j+1}$.
 Denote $a'= a^{\circ p^{i}}$. We need to show that \[a'^{\circ p}*A^{j}\subseteq p^{i+1}A^{j+1}.\]
 By Lemma \ref{14} $a'^{\circ p}*A^{j}=\sum_{i=1}^{p-1}{p\choose i}e_{i}'(a', A^{j})+e_{p}'(a', A^{j})$. Notice that ${p\choose i}$ is divisible by $p$ for $0<i<p$ and by the inductive assumption $a'*A^{j}\subseteq p^{i}A^{j+1}$, hence  $\sum_{i=1}^{p-1}{p\choose i}e_{i}'(a', A^{j})\subseteq p^{i+1}A^{j+1}$. Observe also that $e_{p}'(a', A^{j})=
e_{p-1}'(a', a'*A^{j})\subseteq e_{p-1}'(a', p^{i}A^{j+1})\subseteq e_{p-2}'(a', p^{i}\cdot (a'*A^{j+1}))\subseteq 
 p^{2i} A^{j+1}\subseteq p^{i+1}A^{j+1}$ by the inductive assumption (note that $p\geq 3$ as otherwise the property $1'$ does not hold).  
\end{proof}

\begin{proposition}\label{bbb} Let $A$ be a brace of cardinality $p^{n}$ and such that for some $c$ we have  
$A^{c}\in pA$ and $A*(A*\cdots A*ann(p^{i}))\subseteq  p\cdot ann(p^{i})$ where $A$ appears $c$-times in this expression.
 
 Let $P=(A/ann(p^{2k}),+, \bullet )$ be the pre-Lie ring constructed in Theorem \ref{1}. Denote,   
 ${\bar {ann} }(p^{i}):\{a\in P: p^{i}a=0\}$.
Then,   $P\bullet(P\bullet (\cdots (P\bullet {\bar {ann}} (p^{i}))))\subseteq p\cdot {\bar {ann} }(p^{i})$, where $P$ appears $c$ times in this expression.
\end{proposition}
\begin{proof}
Denote $Q_{t,j}=A*(A*(\cdots (A*ann(p^{j}))))$ where $A$ appears  
$t$ times in this expression.
 
We will show that 
 \[a^{\circ p^{i}}*Q_{t,j}\subseteq p^{i}Q_{t+1,j}\] for $i,j>0$, $t\geq 0$.  Fix $t$. For this given $t$, we proceed by induction on $i$.
  For $i=0$ the result follows since $a*Q_{t,j}\subseteq Q_{t+1,j}=p^{0}Q_{t+1,j}$
 Suppose that the result holds for some $i\geq 0$, so 
  $a^{\circ p^{i}}*A^{j}\subseteq p^{i}A^{j+1}$.
 Denote $a'= a^{\circ p^{i}}$. We need to show that \[a'^{\circ p}*Q(t,j)\subseteq p^{i+1}Q_{t+1,j}.\]

 By Lemma \ref{14}, \[a'^{\circ p}*Q_{t,j}=\sum_{i=1}^{p-1}{p\choose i}e_{i}'(a', Q_{t,j})+e_{p}'(a', Q_{t,j}).\] Notice that ${p\choose i}$ is divisible by $p$ for 
 $0<i<p$ and by the inductive assumption $a'*Q_{t,j}\subseteq p^{i}Q_{t+1,j}$,
 hence  $p^{i-1}\sum_{i=1}^{p-1}{p\choose i}e_{i}'(a', Q_{t,j})\subseteq p^{i+1}Q_{t+1,j}$.

 Observe also that \[e_{p}'(a', Q_{t,j})=e_{p-1}'(a', a'*Q_{t,j})\subseteq e_{p-1}'(a', p^{i}*Q_{t+1,j})\subseteq p^{i+1}*Q_{t+1,j}\] by the inductive assumption (and by assumption 
 $A*(A*\cdots (A*ann(p^{j})))\subseteq  p\cdot ann(p^{j})$).

 Notice that $p^{k}A=\{a^{\circ p^{k}}:a\in A\}$ by  Corollary \ref{bezdowodu} and Lemma \ref{xyz}. Let $d\in p^{k}A$, then $d=a^{\circ k}$ for some $a\in A$. By the above $[d]\odot [Q_{t,j}]\subseteq [Q_{t+1, j}]$ therefore, 
\[P\bullet(P\bullet (\cdots (P\bullet ann_{P}(p^{i}))))\subseteq 
[Q_{c,i}]\subseteq 
p\cdot {\bar {ann }}(p^{i}),\] where $P$ appears $c$ times in this expression.
 This concludes the proof. 
\end{proof}

\section{A passage from pre-Lie rings to  braces}

In this section we denote  $[a]=[a]_{ann (p^{2k})}$.
 
  From Lemma \ref{xyz} it follows that braces satisfying properties $1'$ and $1''$ satisfy properties $1$ and $2$ introduced in the revised version of paper \cite{paper2}. 
 Therefore the following Theorem \ref{71}  is a special case of the result from \cite{paper2} (which was proved in the appendix in \cite{paper2}).

\begin{theorem}\label{71}
 Let $A$ be a brace which satisfy properties $1'$ and $1''$. Let $k$ be such that $p^{k(p-1)}A=0$.
 Let $\rho ^{-1}:pA\rightarrow A$ be any function such that $p\cdot \rho ^{-1}(x)=x$ for each $x\in pA$.
 Let $a,b\in A$, then $[a]*[b]$ is obtained by applying
 operations $+$, $\odot $, 
 to elements $[a]$ and $[b]$ and $\rho ^{-1}$ to elements from $pA$, and by taking cosets $x\rightarrow [x]$ for $x\in A$. Moreover, 
 the order of applying these operation depends only on the additive group of $A$ (it does not depend on $a$ and $b$ and it does not depend on any properties of $*$).  Moreover, the result
 does not depend of the choice of function $\rho ^{-1}$. 
\end{theorem}

 We can now prove our main result:

\begin{theorem}\label{nareszcie} Let $p>3$ be a prime number. Let $n,k$ be natural numbers.
 Let $A$ be a brace of cardinality $p^{n}$ satisfying properties $1'$ and $1''$. Suppose that $p^{k(p-1)}A=0$. Let $\wp^{-1}: p^{k}A\rightarrow A$ be a function satisfying $p\cdot \wp^{-1}(x)=x$ for all $x\in p^{k}A$.  Let $(A/ann(p^{2k}), +, \bullet)$ be the pre-Lie ring constructed in Theorem \ref{1} from the brace $A$. Then the following holds. 
\begin{enumerate}
\item   
\[p^{2k}[a]\odot [b]=q'([a], [b])\]
 where $q'([a],[b])$ is an element obtained by applying
 operations $+$, $\bullet  $ 
 to elements $[a],[b]$ and the order of applying these operation depends only on the additive group of $A$ (and it does not depend on $a,b$ or $* $).

\item The  brace $A/{ann(p^{4k})}$ can be  recovered from the pre-Lie ring $(A/ann(p^{2k}), \bullet, +)$ 
 by applying operations $+, \bullet$ and $\rho ^{-1}$, and the operation of taking cosets,  and the order of applying these operations depends only on the additive group of $A$, and the result does not depend on the choice of function $\rho ^{-1}$ .
\end{enumerate}
\end{theorem} 
\begin{proof}  {\em An introductory  part.} We will express $[x] \odot [y]$ by using the pre-Lie operation $\bullet $ and $+$ applied
to some copies of elements $[x], [y]$, for $x, y\in A$ (in a way which only depends on the
additive group $(A, +)$). Observe that $B=p^{k}A$ is a brace such that $B^{[p-1]}=0$, since $p^{k}A$ is an ideal in $A$ and $(p-1)k\geq \alpha $.

 Observe that brace $B=p^{k}A$ has strong nilpotency index $k<p$. 
By using the same proof as in \cite{passage} we see that  brace $B=p^{k}A$
  is obtained as a group of flows of a left nilpotent  pre-Lie ring $(B, +, \cdot )$. Because brace $B$ is strongly nilpotent of strong nilpotency order $k<p$ we can use 
the same proofs as in \cite{passage} (or we can prove this fact  by  using a more recent and more general result proved in \cite{ST} that every brace of the left nilpotency index less than $p$ can be obtained as the group of flows of some left nilpotent pre-Lie algebra).

Moreover, because $B^{[p-1]}=0$ we obtain that 
for $x,y\in B,$ 
\[\sum_{i=0}^{p-2}\xi ^{p-1-i}((\xi ^{i}x)* y)=(p-1)x\cdot y.\]

 Therefore for $a,b\in A$, for $x=p^{k}a$, $y=p^{k}b$ we get that 
 \[(p^{k}a)*(p^{k}b)=q(a,b)\] where $q(a,b)$  is obtained by the formula for the group of flows (which can be expressed by using operations $\cdot $ and $+$)  applied to elements $p^{k}a$ and $p^{k}b$.
 This shows that $[a]\odot [b]=[\wp^{-1}(\wp^{-1}(q(p^{k}a, p^{k}b)))]$.
 Recall that $\wp ^{-1}:p^{k}A\rightarrow A$ is obtained by applying $k$-times $\rho ^{-1}$.
$ $

{\em Part $1$.} Let notation be as in Part $1$ above. 
Let $A/ann (p^{i})$ be the pre-Lie ring constructed in Theorem \ref{1} from our brace $A$.
 Then by the  definition of $\bullet $ we have
  \[p^{2k}[a]\bullet [ b]= [(p^{k}a)\cdot (p^{k}b)]=[p^{k}a]\cdot [p^{k}b].\]

By taking the factor brace $A/p^{2k}A$ we obtain that 
 $p^{2k}[a*b]$ is obtained by applying  operations $\cdot $ and $+$ to elements $[p^{k}a], [p^{k}b]$ and hence it is obtained by applying operations $\bullet $ and $+$ to elements $[a], [b]$.
 We write it as 
\[p^{2k}[a]\odot [b]=[(p^{k}a)]*[(p^{k}b)]=q'([a],[b]),\]
 where $q'([a],[b])$ is obtained by applying operations $\bullet $ and $+$ to elements $[a], [b]$.

 This concludes the proof of Part $1$.

$ $

{\em Part $2$}  is obtained by applying Part $1$ and observing that 
\[p^{2k}[a]\odot [b]=[(p^{k}a)*(p^{k}b)]=q'([a],[b])\]
 implies
\[[[a]\odot [b]]_{ann(p^{2k})}=[\wp^{-1
}(\wp^{-1}(q'([a],[b])))]_{ann(p^{2k})}.\]
 This can be seen by multiplying both sides by $p^{2k}$. 

 Therefore \[[[a]\odot [b]]_{ann (p^{4k})}=w([a], [b])\]
 where $w([a]_{ann (p^{4k})},[b]_{ann (p^{4k})})$ is an element obtained by applying
 operations $+$, $\bullet  $, $\rho ^{-1}$
 to elements $[a],[b]$ and the order of applying these operation depends only on the additive group of $A$. It also applies the  operation of taking cosets $x\rightarrow [x]_{ann(p^{2})}$ for $x\in A/ann(p^{2k})$.  Recall also that the result
 does not depend of the choice of function $\rho ^{-1}$ used to construct it (and $\rho ^{-1}$ is always applied to elements from $pA$). 
 Moreover, $w([a], [b]) =[\wp^{-1}(\wp^{-1}(q'([a],[b])))]_{ann(p^{2k})}$.

 Denote $x=[a], y=[b]$ then $x,y\in A/ann(p^{2k})$. 
Observe that 
\[[[\wp^{-1}((p^{k}a)*b)]]_{ann(p^{2k})}=[[a]\odot [b]]_{ann(p^{2k})}=[\wp^{-1}(p^{k}x*y)]\]
 This can be seen by multyplying both sides by $p^{2k}$.
 Therefore, 
\[[[a]\odot [b]]_{ann(p^{2k})}=[x]_{ann(p^{2k})}\odot [y]_{ann(p^{2k})}.\]

 we obtain that if $x,y\in A/ann(p^{2k})$ then 
$[x]_{ann(p^{2k})}\odot [y]_{ann(p^{2k})}$ is obtained by taking
operations $+$, $\bullet  $, $\rho ^{-1}$ (and operation of taking cosets)
 to elements $x,y$ and the order of applying these operation depends only on the additive group of $A$. The result now follows from 
 Theorem \ref{71}.
\end{proof}

{\bf Proof of Theorem \ref{main}}.  It follows from Theorem \ref{nareszcie} (2).

\begin{lemma}\label{s} Let $A$ be a brace of cardinality $p^{n}$ for some prime number $p$ and some natural number $n$. Suppose  that the additive group of $A$ is is a direct sum of some number of cyclic groups of cardinality $p^{\alpha }$ for some natural number $\alpha  $. Suppose that $A$ satisfies property $1'$, then $A$ satisfies property $1''$. 
\end{lemma}
\begin{proof} Suppose that  the additive group of $A$ be a direct sum of cyclic 
 groups of cardinality $p^{\alpha }$ for some $\alpha $. Then $ann(p^{i})=p^{\alpha -i}A$, for each $i$. Therefore, 
   $ a*(a* \cdots a*ann(p^{i}))=a*(a* \cdots (a*p^{\alpha -i}A))\in p\cdot p^{\alpha -i}A=p^{\alpha +1-i}A=ann(p^{i-1})$.
\end{proof}

{\bf Proof of Corollary \ref{uniform}.} It follows from Theorem \ref{main} and from Lemma \ref{s}. 
 Notice that Proposition \ref{bbb} applies for braces satisfying assumptions of Corollary \ref{uniform}. 
$ $

{\bf Acknowledgments.} The author acknowledges support from the
EPSRC programme grant EP/R034826/1 and from the EPSRC research grant EP/V008129/1. 
 The author is very  grateful to Michael West for his help with writing the introduction. The author is very greateful to Bettina Eick and Efim Zelmanov for answering her questions about what is known about extensions of Lazard's correspondence in group theory.

\end{document}